\documentclass[12pt]{amsart}
\usepackage{amsthm}
\usepackage{amssymb}
\usepackage{geometry}
\usepackage{enumerate}
\usepackage{setspace}
\usepackage{pgfkeys}
\usepackage{upgreek}
\usepackage{tikz}
\usepackage{tikz-cd}
\usetikzlibrary{matrix, arrows}
\usepackage[unicode=true]
 {hyperref}
\hypersetup{
colorlinks=true,
urlcolor=black,
citecolor=blue,
linkcolor=blue,
}

\newtheorem{thm}{Theorem}[section]
\newtheorem{prop}[thm]{Proposition}
\newtheorem{lem}[thm]{Lemma}

\theoremstyle{definition}
\newtheorem{definition}[thm]{Definition}

\theoremstyle{remark}
\newtheorem{remark}[thm]{Remark}

\numberwithin{equation}{section}

\begin{document}

\large 

\title[The Square Root of the Inverse Different]{Galois Module Structure of the Square Root of the Inverse Different over Maximal Orders}

\author{Cindy (Sin Yi) Tsang}
\address{Department of Mathematics, University of California, Santa Barbara}
\email{cindytsy@math.ucsb.edu}
\urladdr{http://sites.google.com/site/cindysinyitsang/} 

\date{\today}

\begin{abstract}Let $K$ be a number field with ring of integers $\mathcal{O}_K$ and let $G$ be a \mbox{finite group of} odd order. Given a $G$-Galois $K$-algebra $K_h$, let $A_h$ be the square root of the inverse different of $K_h/K$, which exists by Hilbert's formula. If $K_h/K$ is weakly ramified, then $A_h$ is locally free over $\mathcal{O}_{K}G$ by a result of B. Erez, in which case it determines a class in the locally free class group $\mbox{Cl}(\mathcal{O}_KG)$ of $\mathcal{O}_KG$. Such a class in $\mbox{Cl}(\mathcal{O}_KG)$ is said to be $A$-realizable, and tame $A$-realizable if $K_h/K$ is tame. Let $\mathcal{A}(\mathcal{O}_KG)$ and $\mathcal{A}^t(\mathcal{O}_KG)$ denote the sets of all $A$-realizable classes and tame $A$-realizable classes, respectively. For $G$ abelian, we will show that the two sets $\mathcal{A}(\mathcal{O}_KG)$ and $\mathcal{A}^t(\mathcal{O}_KG)$ are equal when extended scalars to the maximal order in $KG$.
\end{abstract}

\maketitle

\tableofcontents

%--- Introduction ---%
\section{Introduction}\label{intro}

Let $K$ be a number field with ring of integers $\mathcal{O}_K$ and let $G$ be a \mbox{finite group.} Let $\Omega_K$ denote the absolute Galois group of $K$ and let $\Omega_K$ act \mbox{trivially on $G$} (on the left). Then, the set of all isomorphism classes of $G$-Galois $K$-algebras (see Section~\ref{Galgebra} for a brief review) is in bijection with the pointed set 
\[
H^1(\Omega_K,G)=\mbox{Hom}(\Omega_K,G)/\mbox{Inn}(G).
\]
Given $h\in H^1(\Omega_K,G)$, we will write $K_h$ for a Galois algebra \mbox{representative. If} the inverse different of $K_h/K$ has a square root, then we will denote it by $A_h$. We will study the Galois module structure of $A_h$ in this paper.

In what follows, assume that $G$ has odd order. Then, for any $h\in H^1(\Omega_K,G)$, the inverse different of $K_h/K$ has a square root by Proposition~\ref{Hformula} below.

\begin{prop}\label{Hformula}Let $p$ be a prime number and let $F/\mathbb{Q}_p$ be a finite extension. Let $N/F$ be a finite Galois extension with different ideal $\mathfrak{D}_{N/F}$ and let $\pi_N$ \mbox{be a} uniformizer in $N$. Then, we have $\mathfrak{D}_{N/F}=(\pi_N)^{v_N(\mathfrak{D}_{N/F})}$ for
\begin{equation}\label{hilbert}
v_N(\mathfrak{D}_{N/F})=\sum_{n=0}^{\infty}(|\mbox{Gal}(N/F)_n|-1),
\end{equation}
where $\mbox{Gal}(N/F)_n$ is the $n$-th ramification group of $N/F$ in lower numbering.
\end{prop}
\begin{proof}
See \cite[Chapter IV, Proposition 4]{S}, for example. We remark that (\ref{hilbert}) is also known as Hilbert's formula.
\end{proof}

If $h\in H^1(\Omega_K,G)$ is such that $K_h/K$ is \emph{weakly ramified} (see Definition~\ref{ramification}), then $A_h$ is locally free over $\mathcal{O}_KG$ by \cite[Theorem 1 in Section 2]{E} and it deter- mines a class $\mbox{cl}(A_h)$ in the locally free class group $\mbox{Cl}(\mathcal{O}_KG)$ of $\mathcal{O}_KG$. \mbox{Such a} class in $\mbox{Cl}(\mathcal{O}_KG)$ is said to be \emph{$A$-realizable}, and \emph{tame $A$-realizable} \mbox{if $K_h/K$ is} tame. We will write
\[
\mathcal{A}(\mathcal{O}_KG):=\{\mbox{cl}(A_h):h\in H^1(\Omega_K,G)\mbox{ with $K_h/K$ weakly ramified}\}
\]
for the set of all $A$-realizable classes in $\mbox{Cl}(\mathcal{O}_KG)$, and
\[
\mathcal{A}^t(\mathcal{O}_KG):=\{\mbox{cl}(A_h):h\in H^1(\Omega_K,G)\mbox{ with $K_h/K$ tame}\}
\]
for the subset of $\mathcal{A}(\mathcal{O}_KG)$ consisting of the tame $A$-realizable classes. 

In what follows, assume that $G$ is abelian in addition to having odd order. In \cite[(12.1) and Theorem 1.3]{T}, the author gave a complete characterization of the set $\mathcal{A}^t(\mathcal{O}_KG)$ and showed that $\mathcal{A}^t(\mathcal{O}_KG)$ is a subgroup of $\mbox{Cl}(\mathcal{O}_KG)$. In \cite[Theorem 1.6]{T}, the author further showed that for \mbox{any $h\in H^1(\Omega_K,G)$ with} $K_h/K$ weakly ramified, if the wildly ramified primes of $K_h/K$ satisfy certain extra hypotheses, then we have $\mbox{cl}(A_h)\in\mathcal{A}^t(\mathcal{O}_KG)$. It is then natural to ask whether the two sets $\mathcal{A}(\mathcal{O}_KG)$ and $\mathcal{A}^t(\mathcal{O}_KG)$ are in fact equal.

In this paper, we will prove that $\mathcal{A}(\mathcal{O}_KG)$ and $\mathcal{A}^t(\mathcal{O}_KG)$ become equal once we extend scalars to the maximal $\mathcal{O}_K$-order $\mathcal{M}(KG)$ in $KG$. More precisely, let $\mbox{Cl}(\mathcal{M}(KG))$ denote the locally free class group of $\mathcal{M}(KG)$ and let
\[
\Psi:\mbox{Cl}(\mathcal{O}_KG)\longrightarrow\mbox{Cl}(\mathcal{M}(KG))
\]
be the natural homomorphism afforded by extension of scalars. From another result \cite[Theorem 1.7]{T} of the author, we have $\Psi(\mathcal{A}(\mathcal{O}_KG))=\Psi(\mathcal{A}^t(\mathcal{O}_KG))$, provided that every prime divisor of $|G|$ is unramified in $K/\mathbb{Q}$. We will show that this additional hypothesis is unnecessary. 

\begin{thm}\label{equal}
Let $K$ be a number field and let $G$ be a finite abelian \mbox{group of} odd order. Then, we have $\Psi(\mathcal{A}(\mathcal{O}_KG))=\Psi(\mathcal{A}^t(\mathcal{O}_KG))$.
\end{thm}
\begin{proof}See Remark~\ref{pfequal} below.
\end{proof}

Using the characterization of the set $\mathcal{A}^t(\mathcal{O}_KG)$ given in \cite[(12.1)]{T}, the proof of Theorem~\ref{equal} reduces to computing the local generators \mbox{of $A_h$ over $\mathcal{O}_KG$ for} each $h\in H^1(\Omega_K,G)$ with $K_h/K$ wildly and weakly ramified. Such generators at the tame primes of $K_h/K$ have already been characterized in \cite[Theorems 10.3 and 10.4]{T}. We will compute these generators at the wild primes of $K_h/K$ in this paper. We will prove (see \mbox{Sections~\ref{notation} and \ref{prereq}} for the notation):

\begin{thm}\label{decomp}Let $p$ be a prime number and let $F/\mathbb{Q}_p$ be a finite extension. Let $G$ be a finite abelian group of odd order and let $h\in H^1(\Omega_F,G)$ be such that $F_h/F$ is wildly and weakly ramified. If $A_h=\mathcal{O}_FG\cdot a$, then
\[
r_G(a)=rag(\gamma)u\Theta^t_*(g)
\]
for some $\gamma\in\mathcal{M}(FG)^\times$, $u\in\mathcal{H}(\mathcal{O}_FG)$, and $g\in \Lambda(FG)^\times$.
\end{thm}

\begin{remark}\label{pfequal}The proof of Theorem~\ref{equal} is that of \cite[Theorem 1.7]{T} verbatim, except we will need to use Theorem~\ref{decomp} above in place of \cite[Theorem 16.1]{T}. \mbox{To avoid} repetition, we will only prove Theorem~\ref{decomp} in this paper.
\end{remark}

%--- Notation and Conventions ---%
\section{Notation and Conventions}\label{notation}

Throughout this paper, unless specified, the symbol $F$ will denote a number field or a finite extension of $\mathbb{Q}_p$ for some prime number $p$. We will \mbox{also fix a} finite abelian group $G$ and we will use the convention that the homomorphisms in the cohomology groups considered are continuous.

For such a field $F$, fix an algebraic closure $F^c$ of $F$ and let $\Omega_F$ denote the Galois group of $F^c/F$. Let $\mathcal{O}_F$ denote the ring of integers in $F$ and write $\mathcal{O}_{F^c}$ for its integral closure in $F^c$. We will let $\Omega_F$ act trivially on $G$ (on the left) and choose a compatible set $\{\zeta_n:n\in\mathbb{Z}^+\}$ of primitive roots of unity in $F^c$. We will also write $\widehat{G}$ for the group of irreducible $F^c$-valued characters on $G$, and $\mathcal{M}(FG)$ for the maximal $\mathcal{O}_F$-order in $FG$.

When $F$ is a finite extension of $\mathbb{Q}_p$, given a finite extension $N/F$, say with uniformizer $\pi_N$ in $N$, let $v_N:N\longrightarrow\mathbb{Z}\cup\{\infty\}$ denote the additive valuation on $N$ for which $v_N(\pi_N)=1$. Given a fractional $\mathcal{O}_N$-ideal $\mathfrak{A}$ in $N$, \mbox{we will also} write $v_N(\mathfrak{A})$ for the unique integer for which $\mathfrak{A}=(\pi_N)^{v_N(\mathfrak{A})}$. Finally, if $N/F$ is Galois, then for each $n\in\mathbb{Z}_{\geq 0}$, let $\mbox{Gal}(N/F)_n$ denote the $n$-th ramification group of $N/F$ in lower numbering.

%--- Prerequisites ---%
\section{Prerequisites}\label{prereq}

In this section, we will define the notation used in the statement of Theorem~\ref{decomp}, in particular \emph{reduced resolvends} and the \emph{modified Stickelberger tranpose}. The former was introduced by L. McCulloh in \cite[Section 2]{M}, where he studied the Galois module structure of rings of integers. The latter \mbox{was intro-} duced by the author in \cite[Section 8]{T} by modifying the definition of the \emph{Stickelberger tranpose} defined by McCulloh in \cite[Section 4]{M}.

%--- Galois Algebras and Resolvends ---%
\subsection{Galois Algebras and Resolvends}\label{Galgebra} We will give a brief review of Galois algebras and resolvends (see \cite[Section 1]{M} for a more detailed discussion). We note that their definitions still make sense even when $G$ is not abelian.

\begin{definition}\label{GaloisAlg}A \emph{Galois algebra over $F$ with group $G$} or \emph{$G$-Galois $F$-algebra} is a commutative semi-simple $F$-algebra $N$ on which $G$ acts (on the left) as a group of automorphisms satisfying $N^{G}=F$ and $[N:F]=|G|$. Two $G$-Galois $F$-algebras are said to be \emph{isomorphic} if there exists an \mbox{$F$-algebra isomorphism} between them which preserves the action of $G$.
\end{definition}

The set of all isomorphism classes of $G$-Galois $F$-algebras may be shown to be in bijective correspondence with the pointed set
\[
H^1(\Omega_F,G):=\mbox{Hom}(\Omega_F,G)/\mbox{Inn}(G).
\]
Since the fixed finite group $G$ is abelian, the isomorphism classes of $G$-Galois $F$-algebras may be identified with the homomorphisms in $\mbox{Hom}(\Omega_F,G)$. More specifically, each $h\in\mbox{Hom}(\Omega_F,G)$ is associated to the $F$-algebra
\[
F_{h}:=\mbox{Map}_{\Omega_F}(^{h}G,F^{c}),
\]
where $^{h}G$ is the group $G$ endowed with the $\Omega_F$-action given by
\[
(\omega\cdot s):=h(\omega)s\hspace{1cm}\mbox{for $s\in G$ and $\omega\in\Omega_F$}.
\]
The $G$-action on $F_h$ is given by
\[
(s\cdot a)(t):=a(ts)\hspace{1cm}\mbox{for $a\in F_h$ and $s,t\in G$}.
\]
Given a set $\{s_i\}$ of coset representatives for $h(\Omega_F)\backslash G$, note that each $a\in F_h$ is uniquely determined by the values $a(s_i)$, and these $a(s_i)$ may be arbitrarily chosen provided that they are fixed by all $\omega\in\ker(h)$. Setting $F^{h}:=(F^{c})^{\ker(h)}$, we see that evaluation at the $s_i$ induces an isomorphism
\[
F_{h}\simeq \prod_{h(\Omega_F)\backslash G}F^{h}
\]
of $F$-algebras. The above isomorphism depends on the choice of the set $\{s_i\}$.

\begin{definition}\label{ramification}
Given $h\in\mbox{Hom}(\Omega_F,G)$, we say that $F_h/F$ or $h$ is \emph{unramified} if $F^h/F$ is unramified. Similarly for \emph{tame}, \emph{wild}, and \emph{weakly ramified}. Recall that a Galois extension over $F$ is said to be weakly ramified if all of its second ramification groups (in lower numbering) are trivial.
\end{definition}

\begin{definition}Given $h\in\mbox{Hom}(\Omega_F,G)$, let $\mathcal{O}^h:=\mathcal{O}_{F^h}$ and define
\[
\mathcal{O}_h:=\mbox{Map}_{\Omega_F}(^hG,\mathcal{O}^h).
\]
If the inverse different of $F^h/F$ is a square, let $A^h$ be its square root and set
\[
A_h:=\mbox{Map}_{\Omega_F}(^hG,A^h).
\]
In the sequel, whenever we write $A_h$ for some $h\in\mbox{Hom}(\Omega_F,G$), we implicitly assume that $A^h$ exists (by Proposition~\ref{Hformula}, this is so when $G$ has odd order).
\end{definition}

\begin{prop}\label{Aexists}Let $F$ be a finite extension of $\mathbb{Q}_p$ and let $h\in\mbox{Hom}(\Omega_F,G)$ be wildly and weakly ramified. Then, the inverse different of $F^h/F$ is a square, and we have $v_{F^h}(A^h)=1-|\mbox{Gal}(F^h/F)_0|$. Moreover, the group $\mbox{Gal}(F^h/F)_0$ is equal to $\mbox{Gal}(F^h/F)_1$ and is elementary $p$-abelian.
\end{prop}
\begin{proof}Let $\mathfrak{D}^h$ denote the different ideal of $F^h/F$. Since $h$ is weakly ramified, by Proposition~\ref{Hformula}, we know that
\[
v_{F^h}(\mathfrak{D}^h)=|\mbox{Gal}(F^h/F)_0|+|\mbox{Gal}(F^h/F)_1|-2.
\]
Now, since $G$ is abelian, by \cite[Chapter IV, Proposition 9, Corollary 2]{S}, we have $\mbox{Gal}(F^h/F)_n=\mbox{Gal}(F^h/F)_{n+1}$ for all $n\in\mathbb{Z}_{\geq 0}$ that is not divisible by 
\[
e_0:=[\mbox{Gal}(F^h/F)_0:\mbox{Gal}(F^h/F)_1].
\]
If $e_0\neq 1$, then $\mbox{Gal}(F^h/F)_1=\mbox{Gal}(F^h/F)_2$ and this is impossible because $h$ is wildly and weakly ramified. Hence, we must have $e_0=1$ and so
\[
\mbox{Gal}(F^h/F)_0=\mbox{Gal}(F^h/F)_1.
\]
We then deduce that $\mathfrak{D}^h$ is a square and that $v_{F^h}(A^h)=1-|\mbox{Gal}(F^h/F)_0|$. Because $\mbox{Gal}(F^h/F)_1/\mbox{Gal}(F^h/F)_2$ is elementary $p$-abelian by \cite[Chapter IV, Proposition 7, Corollary 3]{S} and $\mbox{Gal}(F^h/F)_2=1$ by hypothesis, we then see that the group $\mbox{Gal}(F^h/F)_0$ is elementary $p$-abelian as well.
\end{proof}

Next, consider the $F^c$-algebra $\mbox{Map}(G,F^c)$ on which we let $G$ act via
\[
(s\cdot a)(t):=a(ts)\hspace{1cm}\mbox{for $a\in\mbox{Map}(G,F^c)$ and $s,t\in G$}.
\]
Note that $F_h$ is an $FG$-submodule of $\mbox{Map}(G,F^c)$ for all $h\in\mbox{Hom}(\Omega_F,G)$.

\begin{definition}\label{resolvend}
The \emph{resolvend map} $\mathbf{r}_{G}:\mbox{Map}(G,F^{c})\longrightarrow F^{c}G$ is defined by
\[
\mathbf{r}_{G}(a):=\sum\limits _{s\in G}a(s)s^{-1}.
\]
\end{definition}

The map $\mathbf{r}_{G}$ is clearly an isomorphism of $F^cG$-modules, but not an isomorphism of $F^cG$-algebras because it does not preserve multiplication. Moreover, given $a\in\mbox{Map}(G,F^c)$, we have that $a\in F_h$ if and only if
\begin{equation}\label{resol1}
\omega\cdot\mathbf{r}_{G}(a)=\mathbf{r}_{G}(a)h(\omega)
\hspace{1cm}\mbox{for all }\omega\in\Omega_F.
\end{equation}
The next proposition shows that resolvends may be used to identify elements $a\in F_h$ for which $F_h=FG\cdot a$ or $\mathcal{O}_h=\mathcal{O}_FG\cdot a$ or $A_h=\mathcal{O}_FG\cdot a$. Here $[-1]$ denotes the involution on $F^cG$ induced by the involution $s\mapsto s^{-1}$ on $G$.

\begin{prop}\label{NBG} Let $h\in\mbox{Hom}(\Omega_F,G)$ and let $a\in F_h$ be given. We have
\begin{enumerate}[(a)]
\item $F_h=FG\cdot a$ if and only if $\mathbf{r}_{G}(a)\in (F^{c}G)^{\times}$;
\item $\mathcal{O}_h=\mathcal{O}_FG\cdot a$ with $h$ unramified if and only if $\mathbf{r}_G(a)\in(\mathcal{O}_{F^c}G)^{\times}$;
\item $A_h=\mathcal{O}_FG\cdot a$ if and only if $a\in A_h$ and $\mathbf{r}_G(a)\mathbf{r}_G(a)^{[-1]}\in(\mathcal{O}_FG)^\times$.
\end{enumerate}
\end{prop}
\begin{proof}See \cite[Proposition 1.8 and (2.11)]{M} for (a) and (b), and \cite[Proposition 3.10]{T} for (c).
\end{proof}

We are interested in giving a description of the resolvends $\mathbf{r}_G(a)$ for which $A_h=\mathcal{O}_FG\cdot a$ for a wildly and weakly ramified $h\in\mbox{Hom}(\Omega_F,G)$ when $F$ is a finite extension of $\mathbb{Q}_p$. The next proposition will be a crucial tool; \mbox{it will allow} us to reduce to the case when $F^{h}/F$ is totally ramified.

\begin{prop}\label{factor} Let $F$ be a finite extension of $\mathbb{Q}_p$ and let $h\in\mbox{Hom}(\Omega_F,G)$. 
\begin{enumerate}[(a)]
\item There exists a factorization $h=h^{nr}h^{tot}$ of $h$, with $h^{nr},h^{tot}\in\mbox{Hom}(\Omega_F,G)$, such that $h^{nr}$ is unramified and $F^{h^{tot}}/F$ is totally ramified. Furthermore, if $h$ is wildly and weakly ramified, then so is $h^{tot}$.
\item Assume that $h$ is weakly ramified and let $h=h^{nr}h^{tot}$ be given as in (a). If $\mathcal{O}_{h^{nr}}=\mathcal{O}_FG\cdot a_{nr}$ and $A_{h^{tot}}=\mathcal{O}_FG\cdot a_{tot}$, then there exists $a\in A_h$ such that $A_h=\mathcal{O}_{F}G\cdot a$ and $\mathbf{r}_G(a)=\mathbf{r}_G(a_{nr})\mathbf{r}_G(a_{tot})$.
\end{enumerate}
\end{prop}
\begin{proof}See \cite[Propositions 9.2 and 5.3]{T}.
\end{proof}

%--- Cohomology and Reduced Resolvends ---%
\subsection{Cohomology and Reduced Resolvends}\label{reduced}We will define reduced resolvends and explain how to interpret them as \mbox{functions on characters of $G$.}

Recall that $\Omega_F$ acts trivially on $G$ on the left. Define
\[
\mathcal{H}(FG):=((F^{c}G)^{\times}/G)^{\Omega_F}.
\]
Given a coset $\mathbf{r}_G(a)G\in\mathcal{H}(FG)$, we will denote it by $r_G(a)$, called the \emph{reduced resolvend of $a$}. Now, taking $\Omega_F$-cohomology of the short exact sequence
\begin{equation}\label{es1}
\begin{tikzcd}[column sep=1cm, row sep=1.5cm]
1 \arrow{r} &
G \arrow{r} &
(F^{c}G)^{\times} \arrow{r} &
(F^{c}G)^{\times}/G \arrow{r}&
1
\end{tikzcd}
\end{equation}
yields the exact sequence
\begin{equation}\label{es2}
\begin{tikzcd}[column sep=1cm, row sep=1.5cm]
1 \arrow{r} &
G \arrow{r} &
(FG)^{\times} \arrow{r}{rag} &
\mathcal{H}(FG) \arrow{r}{\delta}&
\mbox{Hom}(\Omega_F,G) \arrow{r}&
1,
\end{tikzcd}
\end{equation}
where exactness on the right follows from the fact that $H^1(\Omega_F,(F^cG)^\times)=1$, which is Hilbert's Theorem 90. Alternatively, a coset $\mathbf{r}_{G}(a)G\in\mathcal{H}(FG)$ is in the preimage of $h\in\mbox{Hom}(\Omega_F,G)$ under $\delta$ if and only if
\[
h(\omega)=\mathbf{r}_G(a)^{-1}(\omega\cdot\mathbf{r}_{G}(a))
\hspace{1cm}\mbox{for all }\omega\in\Omega_F,
\]
which is equivalent to $F_{h}=FG\cdot a$ by (\ref{resol1}) and Proposition~\ref{NBG} (a). Because there always exists an element $a\in F_h$ for which $F_h=FG\cdot a$ by the Normal Basis Theorem, the map $\delta$ is indeed surjective. 

The argument above also shows that
\[
\mathcal{H}(FG)=\{r_{G}(a)\mid F_{h}=FG\cdot a\mbox{ for some }h\in\mbox{Hom}(\Omega_F,G)\}.
\]
Similarly, we may define
\[
\mathcal{H}(\mathcal{O}_FG):=((\mathcal{O}_{F^c}G)^\times/G)^{\Omega_F}.
\]
Then, the argument above together with Proposition~\ref{NBG} (b) imply that
\begin{equation}\label{rrunram}
\mathcal{H}(\mathcal{O}_FG)=\left\lbrace r_G(a)\,\middle\vert\,
\begin{array}{@{}c@{}c}
 \mathcal{O}_h=\mathcal{O}_FG\cdot a\mbox{ for some}\\
\mbox{unramified $h\in\mbox{Hom}(\Omega_F,G)$}
\end{array}
\right\rbrace.
\end{equation}

To view reduced resolvends as functions on characters of $G$, define 
\[
\det:\mathbb{Z}\widehat{G}\longrightarrow\widehat{G};\hspace{1em}\det\Bigg(\sum_{\chi} n_{\chi}\chi\Bigg):=\prod_{\chi}\chi^{n\chi}
\]
and set $S_{\widehat{G}}:=\ker(\det)$. By applying the functor $\mbox{Hom}(-,(F^{c})^{\times})$ to the short exact sequence
\[
\begin{tikzcd}[column sep=1cm, row sep=1.5cm]
0 \arrow{r} &
S_{\widehat{G}} \arrow{r} &
\mathbb{Z}\widehat{G}\arrow{r}[font=\normalsize, auto]{\det} &
\widehat{G} \arrow{r} &
1,
\end{tikzcd}
\]
we obtain the short exact sequence
\begin{equation}\label{es3}
\begin{tikzcd}[column sep=0.45cm, row sep=1.5cm]
1 \arrow{r} &
\mbox{Hom}(\widehat{G},(F^{c})^{\times}) \arrow{r} &
\mbox{Hom}(\mathbb{Z}\widehat{G},(F^{c})^{\times}) \arrow{r}&
\mbox{Hom}(S_{\widehat{G}},(F^{c})^{\times}) \arrow{r}&
1,
\end{tikzcd}
\end{equation}
where exactness on the right follows from the fact that $(F^{c})^{\times}$ is divisible and thus injective. We will identify (\ref{es3}) with (\ref{es1}) as follows.

First, observe that we have canonical identifications
\begin{equation}\label{iden1}
(F^{c}G)^{\times}=\mbox{Map}(\widehat{G},(F^{c})^{\times})
=\mbox{Hom}(\mathbb{Z}\widehat{G},(F^c)^\times).
\end{equation}
The second identification is given by extending the maps $\widehat{G}\longrightarrow(F^c)^\times$ via $\mathbb{Z}$-linearity, and the first identification is induced by characters on $G$ as follows. Each resolvend $\mathbf{r}_{G}(a)\in (F^cG)^\times$ gives rise to a map $\mbox{Map}(\widehat{G},(F^c)^\times)$ given by
\begin{equation}\label{resolvent}
\mathbf{r}_G(a)(\chi):=\sum_{s\in G}a(s)\chi(s)^{-1}\hspace{1cm}\mbox{for }\chi\in\widehat{G}.
\end{equation}
Conversely, given $\varphi\in\mbox{Map}(\widehat{G},(F^c)^\times)$, one recovers  $\mathbf{r}_{G}(a)$ by the formula
\[
a(s):=\frac{1}{|G|}\sum_{\chi}\varphi(\chi)\chi(s)\hspace{1cm}\mbox{for }s\in G.
\]
Since $G=\mbox{Hom}(\widehat{G},(F^{c})^{\times})$ canonically, the third terms in (\ref{es1}) and (\ref{es3}) are naturally identified as well. Taking $\Omega_F$-invariants, we then obtain
\begin{equation}\label{idenH}
\mathcal{H}(FG)=\mbox{Hom}_{\Omega_F}(S_{\widehat{G}},(F^c)^\times).
\end{equation}
Finally, given $c\in (FG)^\times$, we will write $rag(c)$ for its image in $\mathcal{H}(FG)$ under the map $rag$ in (\ref{es2}).

%--- The Modified Stickelberger Transpose ---%
\subsection{The Modified Stickelberger Transpose}\label{StickelTranspose}\label{StickelS} In this subsection, assume further that $G$ has odd order.  Recall from Section~\ref{notation} that we chose a compatible set $\{\zeta_n:n\in\mathbb{Z}^+\}$ of primitive roots of unity in $F^c$. 

\begin{definition}\label{Stickel}For each $\chi\in\widehat{G}$ and $s\in G$, let $
\upsilon(\chi,s)\in \left[\frac{1-|s|}{2},\frac{|s|-1}{2}\right]$
denote the unique integer (note that $|s|$ is odd because $G$ has odd order) such that $\chi(s)=(\zeta_{|s|})^{\upsilon(\chi,s)}$, and define $\langle\chi,s\rangle_{*}:=\upsilon(\chi,s)/|s|$. Extending this definition by $\mathbb{Q}$-linearity, we obtain a pairing $\langle\hspace{1mm},\hspace{1mm}\rangle_*:\mathbb{Q}\widehat{G}\times\mathbb{Q}G\longrightarrow\mathbb{Q}$. The map
\[
\Theta_{*}:\mathbb{Q}\widehat{G}\longrightarrow\mathbb{Q}G;
\hspace{1em}
\Theta_{*}(\psi):=\sum_{s\in G}\langle\psi,s\rangle_{*}s
\]
is called the \emph{modified Stickelberger map}.
\end{definition}

\begin{prop}\label{A-ZG}
Given $\psi\in\mathbb{Z}\widehat{G}$, we have $\Theta_{*}(\psi)\in\mathbb{Z}G$ if and only if $\psi\in S_{\widehat{G}}$.
\end{prop}
\begin{proof}See \cite[Proposition 8.2]{T}.
\end{proof}

Note that $\Omega_F$ acts on $\widehat{G}$  (on the left) canonically via its action \mbox{on the roots} of unity in $F^c$, and recall that $\Omega_F$ acts trivially on $G$ by definition. \mbox{Below, we} define other $\Omega_F$-actions on $G$, one of which will make the $\mathbb{Q}$-linear map $\Theta_*$ preserve the $\Omega_F$-action.

\begin{definition}\label{cyclotomic}
Let $m:=\exp(G)$ and let $\mu_m$ be the group of $m$-th roots of unity in $F^c$. The \emph{$m$-th cyclotomic character of $\Omega_F$} is the homomorphism 
\[
\kappa:\Omega_F\longrightarrow(\mathbb{Z}/m\mathbb{Z})^{\times}
\]
defined by the equations
\[
\omega(\zeta)=\zeta^{\kappa(\omega)}\hspace{1cm}\mbox{for $\omega\in\Omega_F$ and }\zeta\in\mu_m.
\]
For $n\in\mathbb{Z}$, let $G(n)$ be the group $G$ equipped with the $\Omega_F$-action given by
\[
\omega\cdot s:=s^{\kappa(\omega^{n})}\hspace{1cm}\mbox{for $s\in G$ and $\omega\in\Omega_F$}.
\]
\end{definition}

\begin{prop}\label{eqvariant}
The map $\Theta_{*}:\mathbb{Q}\widehat{G}\longrightarrow\mathbb{Q}G(-1)$ preserves $\Omega_F$-action.
\end{prop}
\begin{proof}See \cite[Proposition 8.4]{T}.
\end{proof}

From Propositions~\ref{A-ZG} and \ref{eqvariant}, the map $\Theta_*$ restricts to an $\Omega_F$-equivariant map $S_{\widehat{G}}\longrightarrow\mathbb{Z}G(-1)$. Applying the functor $\mbox{Hom}(-,(F^c)^\times)$, we then obtain an $\Omega_F$-equivariant homomorphism
\[
\Theta_{*}^{t}:\mbox{Hom}(\mathbb{Z}G(-1),(F^{c})^{\times})\longrightarrow\mbox{Hom}(S_{\widehat{G}},(F^{c})^{\times});\hspace{1em}f\mapsto f\circ\Theta_*.
\]
Taking $\Omega_F$-invariants, this yields a homomorphism
\[
\Theta^t_{*}:\mbox{Hom}_{\Omega_F}(\mathbb{Z}G(-1),(F^{c})^{\times})\longrightarrow\mbox{Hom}_{\Omega_F}(S_{\widehat{G}},(F^{c})^{\times}),
\]
called the \emph{modified Stickelberger transpose}. To simplify notation, define
\begin{equation}\label{lambda}
\Lambda(FG):=\mbox{Map}_{\Omega_F}(G(-1),F^c).
\end{equation}
Since there is a natural identification $\Lambda(FG)^\times=
\mbox{Hom}_{\Omega_F}(\mathbb{Z}G(-1),(F^c)^\times)$, we may regard $\Theta_{*}^{t}$ as a homomorphism $\Lambda(FG)^{\times}\longrightarrow\mathcal{H}(FG)$ (recall (\ref{idenH})).

%--- Valuations of Local Wild Resolvents ------------
\section{Valuations of Local Wild Resolvents}\label{computeval}

In order to prove Theorem~\ref{decomp}, we will first prove the following.

\begin{thm}\label{units}Let $F$ be a finite extension of $\mathbb{Q}_p$ and let $h\in\mbox{Hom}(\Omega_F,G)$ be wildly and weakly ramified. If $A_h=\mathcal{O}_FG\cdot a$  (cf. Proposition~\ref{Aexists}), then 
\begin{equation}\label{resolventdef}
(a\mid\chi):=\sum_{s\in G}a(s)\chi(s)^{-1}
\end{equation}
(cf. (\ref{resolvent})) is element of $\mathcal{O}_{F^c}^\times$ for all $\chi\in\widehat{G}$.
\end{thm}

We note that (\ref{resolventdef}) is called the \emph{resolvent of $a$ at $\chi$}. In the rest of this section, let $F$ be a finite extension of $\mathbb{Q}_p$ and let $\zeta=\zeta_p$ be the chosen primitive $p$-th root of unity in $F^c$. We will also use the following notation and conventions.

\begin{definition}\label{defch5}Let $\mathbb{F}_p:=\mathbb{Z}/p\mathbb{Z}$. For each $i\in\mathbb{F}_p$, if $z$ is an element of order $1$ or $p$ in a group, we will write $z^i$ for $z^{n_i}$, where $n_i\in\mathbb{Z}$ is any \mbox{representative of $i$.} If $i\in\mathbb{F}_p^\times$, we will write $i^{-1}$ for its multiplicative inverse in $\mathbb{F}_p^\times$. If $p$ is odd, we will further define $c(i)\in\left[\frac{1-p}{2},\frac{p-1}{2}\right]$ to be the unique \mbox{integer that represents $i$.}
\end{definition}

\begin{definition}\label{Rn}
Notice that $\mathbb{Q}_p$ contains all $(p-1)$-st roots of unity. We will write $\widehat{\mathbb{F}_p^\times}$ for the group of $\mathbb{Q}_p^\times$-valued characters on $\mathbb{F}_p^\times$. Given $\varphi\in\widehat{\mathbb{F}_p^\times}$, we will extend it to a map on $\mathbb{F}_p$ by setting $\varphi(0)=0$. For each $n\in\mathbb{N}$ which divides $p-1$, let $R_n:=(\mathbb{F}_p^\times)^n$ be the subgroup of $\mathbb{F}_p^\times$ consisting of the non-zero $n$-th powers in $\mathbb{F}_p$. 
\end{definition}

%--- Valuations of Gauss Sums over Qp ------------
\subsection{Valuations of Gauss Sums over $\mathbb{Q}_p$} We will prove Theorem~\ref{units} by first computing the valuations of the following Gauss sums.

\begin{definition}\label{Gauss}For each $\varphi\in\widehat{\mathbb{F}_p^\times}$ and $j\in\mathbb{F}_p$, define
\[
G(\varphi,j):=\sum_{k\in\mathbb{F}_p}\varphi(k)\zeta^{jk}.
\]
\end{definition}

\begin{lem}\label{GaussBasic}
For all $\varphi\in\widehat{\mathbb{F}_p^\times}$ and $j\in\mathbb{F}_p^\times$, we have 
\begin{enumerate}[(a)]
\item $G(1,0)=p-1$ and $G(\varphi,0)=0$ if $\varphi\neq 1$;
\item $G(\varphi,j)=\varphi(j)^{-1}G(\varphi,1)$ and $G(1,j)=-1$.
\end{enumerate}
\end{lem}
\begin{proof}The claims in (a) follow from the orthogonality of characters, and both equalities in (b) follow from a simple calculation.
\end{proof}

In view of Lemma~\ref{GaussBasic}, it remains to consider $G(\varphi,1)$ for $\varphi\neq 1$.

\begin{prop}\label{GaussVal}Let $\varphi\in\widehat{\mathbb{F}_p^\times}$ be of order $n\neq 1$. For all $j\in\mathbb{F}_p^\times$, we have
\[
v_{\mathbb{Q}_p(\zeta)}(G(\varphi,j))\geq (p-1)/n.
\]
\end{prop}
\begin{proof}By Lemma~\ref{GaussBasic} (b), it is enough to prove the claim for $j=1$. We will do so by computing the valuation of the sum
\[
S:=\sum_{j\in\mathbb{F}_p}G(\varphi,j)^n.
\]
On one hand, using Definition~\ref{Gauss}, we have
\begin{align*}
S
&=\sum_{j\in\mathbb{F}_p} \sum_{\substack{k_i\in\mathbb{F}_p\\1\leq i\leq n}}\varphi(k_1\cdots k_n)\zeta^{j(k_1+\cdots+k_n)}\\
&=\sum_{\substack{k_i\in\mathbb{F}_p\\1\leq i\leq n}}\varphi(k_1\cdots k_n)\sum_{j\in\mathbb{F}_p}\zeta^{j(k_1+\cdots+k_n)}.
\end{align*}
Since each $\varphi(k_1\cdots k_n)$ is integral and
\[
\sum_{j\in\mathbb{F}_p}\zeta^{j(k_1+\cdots+k_n)}=\begin{cases}
p&\mbox{ if }k_1+\cdots+k_n=0\\
0&\mbox{ otherwise},
\end{cases}
\]
the sum $S$ is the product of $p$ and an element of non-negative valuation, so
\begin{equation}\label{S1}
v_{\mathbb{Q}_p(\zeta)}(S)\geq v_{\mathbb{Q}_p(\zeta)}(p)=p-1.
\end{equation}
On the other hand, notice that $G(\varphi,0)=0$ by Lemma~\ref{GaussBasic} (a) because $\varphi\neq 1$. Using Lemma~\ref{GaussBasic} (b) and the fact that $\varphi$ has order $n$, we then see that
\[
S=\sum_{j\in\mathbb{F}_p^\times}\varphi(j)^{-n}G(\varphi,1)^n=(p-1)G(\varphi,1)^n.
\]
Since $p-1$ has valuation zero, this shows that
\begin{equation}\label{S2}
v_{\mathbb{Q}_p(\zeta)}(S)=n\cdot v_{\mathbb{Q}_p(\zeta)}(G(\varphi,1)).
\end{equation}
The desired inequality now follows from (\ref{S1}) and (\ref{S2}).
\end{proof}

We will also need the following proposition.

\begin{prop}\label{gG}Let $\varphi\in\widehat{\mathbb{F}_p^\times}$ be of order $n\neq 1$. For all $j\in\mathbb{F}_p^\times$, we have
\[
\sum_{l=1}^{n-1}G(\varphi^l,j)=1+n\sum_{k\in R_n}\zeta^{jk}.
\]
\end{prop}
\begin{proof}First of all, we have
\[
\sum_{l=0}^{n-1}G(\varphi^l,j)
=\sum_{l=0}^{n-1}\sum_{k\in\mathbb{F}_p}\varphi^l(k)\zeta^{jk}
=\sum_{k\in\mathbb{F}_p}\zeta^{jk}\sum_{l=0}^{n-1}\varphi^l(k).
\]
Observe that $\ker(\varphi)=R_n$ because $\varphi$ has order $n$. In particular, we may regard $1,\varphi,\cdots,\varphi^{n-1}$ as the distinct characters on $\mathbb{F}_p^\times/R_n$. By the orthogonality of characters, we see that
\[
\sum_{l=0}^{n-1}\varphi^l(k)=\begin{cases}
n&\mbox{if }k\in R_n\\
0&\mbox{otherwise}.
\end{cases}
\]
It follows that 
\[
\sum_{l=0}^{n-1}G(\varphi^l,j)=n\sum_{k\in R_n}\zeta^{jk}.
\]
Since $G(1,j)=-1$ by Lemma~\ref{GaussBasic} (b), the claim now follows.
\end{proof}

\subsection{Proof of Theorem~\ref{units}} First, we make the following observation, and we will consider the special case when $h\in\mbox{Hom}(\Omega_F,G)$ is wildly and weakly ramified with $F^h/F$ totally ramified (cf. Proposition~\ref{factor}).

\begin{lem}\label{prelimval}Let $h\in\mbox{Hom}(\Omega_F,G)$ be given. If $A_h=\mathcal{O}_FG\cdot a$, then
\[
v_{N}((a\mid\chi^{-1}))=-v_{N}((a\mid\chi))
\]
for all $\chi\in\widehat{G}$. In particular, we have $v_F((a\mid 1))=0$. Here $N/F$ is any finite extension that contains $(a\mid\chi)$ for all $\chi\in\widehat{G}$.
\end{lem}
\begin{proof}This follows from the observation that 
\[
(a\mid\chi)(a\mid\chi^{-1})=\mathbf{r}_G(a)\mathbf{r}_G(a)^{[-1]}(\chi)
\]
(cf. (\ref{resolvent}) and recall that $[-1]$ denotes the involution on $F^cG$ induced by the involution $s\mapsto s^{-1}$ on $G$), which lies in $\mathcal{O}_N^\times$ by Proposition~\ref{NBG} (c).
\end{proof}

We note that the next proposition is a generalization of \cite[Theorem 15.4]{T}.

\begin{prop}\label{valtotram}Let $h\in\mbox{Hom}(\Omega_F,G)$ be wildly and weakly ramified and be such that  $F^h/F$ is totally ramified. Then, there exists $a\in A_h$ \mbox{(cf. Proposition} \ref{Aexists}) such that $A_h=\mathcal{O}_FG\cdot a$ and $(a\mid\chi)\in\mathcal{O}_{F^c}^\times$ for all $\chi\in\widehat{G}$.
\end{prop}
\begin{proof}From Proposition~\ref{Aexists}, we know that
\[
v_{F^h}(A^h)\equiv 1\hspace{1cm}\mbox{(mod $|\mbox{Gal}(F^h/F)_1|)$}.
\]
Then, by \cite[Theorem 1.1]{HJ}, we have $A^h=\mathcal{O}_F\mbox{Gal}(F^h/F)\cdot\alpha$ for some $\alpha\in A^h$. Define $a\in\mbox{Map}(G,F^c)$ by setting
\[
a(s):=
\begin{cases}
\omega(\alpha) &\mbox{if $s=h(\omega)$ for $\omega\in\Omega_F$}\\
0 & \mbox{otherwise}.
\end{cases}
\]
It is not hard to see that $a$ is well-defined and that $A_h=\mathcal{O}_FG\cdot a$. It remains to show that $(a\mid\chi)\in\mathcal{O}_{F^c}^\times$ for all $\chi\in\widehat{G}$.

To that end, first observe that for each $\chi\in\widehat{G}$, we have
\begin{equation}\label{achi}
(a\mid\chi)=\sum_{s\in h(\Omega_F)}a(s)\chi(s)^{-1}.
\end{equation}
Note that $h(\Omega_F)\simeq\mbox{Gal}(F^h/F)$ and $\mbox{Gal}(F^h/F)=\mbox{Gal}(F^h/F)_0$ since $F^h/F$ is totally ramified. Because $\mbox{Gal}(F^h/F)_0$ has exponent $p$ by Proposition~\ref{Aexists}, so does $h(\Omega_F)$. In particular, the resolvent $(a\mid\chi)$ lies in $F^h(\zeta)$.

If $p=2$, then $(a\mid\chi)=(a\mid\chi^{-1})$ and hence $(a\mid\chi)\in\mathcal{O}_{F^h(\zeta)}^\times$ by Lemma~\ref{prelimval}. If $p$ is odd and $[F(\zeta):F]$ is even, then $(a\mid\chi)\in\mathcal{O}_{F^h(\zeta)}^\times$ by \cite[Theorem 15.4]{T}.

If $p$ and $[F(\zeta):F]$ are both odd, then $\mbox{Gal}(F(\zeta)/F)\simeq R_n$ for some $n\in\mathbb{N}$ dividing $p-1$ with $n\neq 1$. Now, suppose on the contrary that $(a\mid\chi)\notin\mathcal{O}_{F^h(\zeta)}^\times$ for some $\chi\in\widehat{G}$. By Lemma~\ref{prelimval}, we know that $\chi\neq 1$ and we may assume that $v_{F^h(\zeta)}((a\mid\chi))>0$. For each $s\in h(\Omega_F)$, let  $j_s\in\mathbb{F}_p$ be such that $\chi^{-1}(s)=\zeta^{j_s}$. Let $\varphi\in\widehat{\mathbb{F}_p^\times}$ be any character of order $n$ and consider the sum
\[
S:=\sum_{s\in h(\Omega_F)}a(s)\sum_{l=1}^{n-1}G(\varphi^l,j_s).
\]
We will obtain a contradiction by computing the valuation of $S$. 

First, let $e,d,r\in\mathbb{N}$ be such that the numbers in the  diagram below represent the ramification indices of the extensions. 
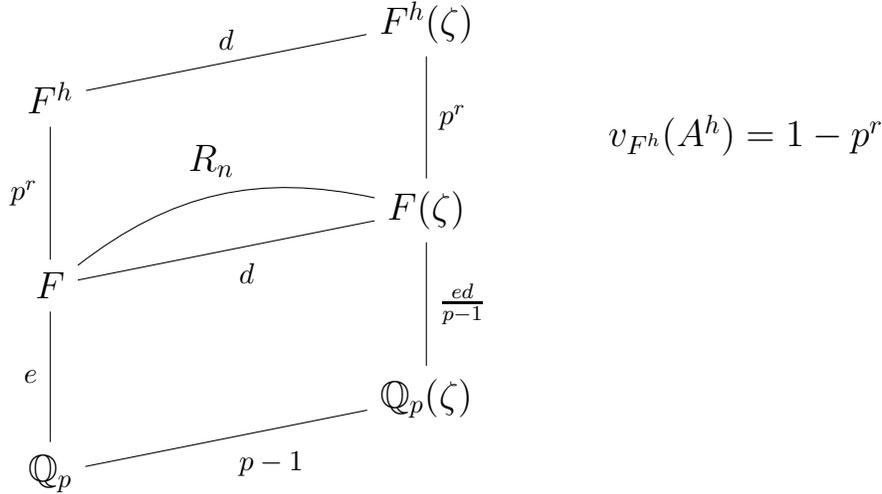
\[
\begin{tikzpicture}[baseline=(current bounding box.center)]
\node at (0,0) [name=F] {$F$};
\node at (5,1) [name=F'] {$F(\zeta)$};
\node at (0,2.5) [name=L] {$F^h$};
\node at (5,3.5) [name=L'] {$F^h(\zeta)$};
\node at (0,-2.5) [name=Q] {$\mathbb{Q}_p$};
\node at (5,-1.5) [name=Q'] {$\mathbb{Q}_p(\zeta)$};
\node at (9.25,2) {$v_{F^h}(A^h)=1-p^r$};
\path[font=\small]
(F') edge node[auto] {$\frac{ed}{p-1}$} (Q')
(Q) edge node[auto] {$e$} (F)
(Q') edge node[auto] {$p-1$} (Q)
(F') edge node[auto] {$d$} (F)
(F) edge node[auto] {$p^r$} (L)
(L') edge node[above] {$d$} (L)
(L') edge node[auto] {$p^r$} (F');
\draw[bend left=25] (F) edge node[pos=0.6,auto] {$R_n$} (F');
\end{tikzpicture}
\]
Note that $v_{F^h}(A^h)=1-p^r$ because $v_{F^h}(A^h)=1-|\mbox{Gal}(F^h/F)_0|$ by Proposition~\ref{Aexists}. For each $s\in h(\Omega_F)$ and  $l=1,\dots,n-1$, if $j_s=0$, then $G(\varphi^l,j_s)=0$ by Lemma~\ref{GaussBasic} (a). If $j_s\neq 0$, then using Proposition~\ref{GaussVal}, we obtain
\begin{align}\label{dval}
v_{F^h(\zeta)}(a(s)G(\varphi^l,j_s))
&\geq d(1-p^r)+\frac{ed}{p-1}\cdot p^r\cdot\frac{p-1}{n}\\
&=dp^r\Big(\frac{e}{n}-1\Big)+d.\notag
\end{align}
Note that $nd\leq n|R_n|=p-1$ and that $p-1\leq ed$ by the multiplicativity of ramification indices. It follows that $n\leq e$ and so (\ref{dval}) is positive. We then deduce that $S$ has positive valuation.

Next, let $H_0$ denote the subgroup of $h(\Omega_F)$ consisting of the elements $s$ for 

\noindent which $j_s=0$. Then, for all $s\in H_0$, we have $G(\varphi^l,j_s)=0$ for $l=1,\dots,n-1$ by Lemma~\ref{GaussBasic} (a). Using Proposition~\ref{gG}, we may then rewrite
\[
S=\sum_{s\in h(\Omega_F)}a(s)\left(1+n\sum_{k\in R_n}\zeta^{j_sk}\right)-\sum_{s\in H_0}a(s)\left(1+n\sum_{k\in R_n}\zeta^{(0)k}\right).
\]
Recall that $\chi^{-1}(s)=\zeta^{j_s}$ by definition. By (\ref{achi}), the above the simplifies to
\begin{align*}
S&=\sum_{s\in h(\Omega_F)}a(s)+n\sum_{k\in R_n}\sum_{s\in h(\Omega_F)}a(s)\chi^k(s)^{-1}-p\sum_{s\in H_0}a(s)\\
&=(a\mid 1)+n\sum_{k\in R_n}(a\mid\chi^k)-p\sum_{s\in H_0}a(s).
\end{align*}
Since $[F(\zeta):F]$ and $[F^h:F]$ are coprime, there is a canonical isomorphism
\[
\mbox{Gal}(F^h(\zeta)/F)\simeq\mbox{Gal}(F(\zeta)/F)\times\mbox{Gal}(F^h/F).
\]
For each $k\in R_n$, let $\omega_k$ be the element in $\mbox{Gal}(F(\zeta)/F)$ such that $\omega_k(\zeta)=\zeta^k$ and set $\widetilde{\omega_k}:=\omega_k\times\mbox{id}_{F^h}$. Then, clearly $(a\mid\chi^{k})=\widetilde{\omega_k}((a\mid\chi))$ and so 
\[
v_{F^h(\zeta)}((a\mid\chi^k))=v_{F^h(\zeta)}((a\mid\chi)),
\]
which is positive by assumption. We have already shown that $S$ has positive valuation. Since $(a\mid 1)$ has valuation zero by Lemma~\ref{prelimval}, we deduce that
\[
v_{F^h}\left(p\sum_{s\in H_0}a(s)\right)=0.
\]
Since $a\in A_h$, this in turn implies that
\begin{align*}
0&\geq v_{F^h}(p)+v_{F^h}(A^h)\\
&=ep^r+(1-p^r)\\
&=p^r(e-1)+1,
\end{align*}
which is impossible. Hence,we must have $(a\mid\chi)\in\mathcal{O}_{F^h(\zeta)}^\times$ for all $\chi\in\widehat{G}$.
\end{proof}

\begin{proof}[Proof of Theorem~\ref{units}]Let $h=h^{nr}h^{tot}$ be a factorization of $h$ as in Proposition~\ref{factor} (a). Since $F^{h^{tot}}/F$ is wildly, weakly, and totally ramified, by Proposition~\ref{valtotram}, there exists $a_{tot}\in A_{h^{tot}}$ such that $A_{h^{tot}}=\mathcal{O}_FG\cdot a_{tot}$ and
\begin{equation}\label{val2}
(a_{tot}\mid\chi)\in\mathcal{O}_{F^c}^\times\hspace{1cm}\mbox{for all }\chi\in\widehat{G}.
\end{equation}
On the other hand, by a classical theorem of E. Noether (alternatively, by \cite[Proposition 5.5]{M}), there exists $a_{nr}\in\mathcal{O}_{h^{nr}}$ such that $\mathcal{O}_{h^{nr}}=\mathcal{O}_FG\cdot a_{nr}$, and
\begin{equation}\label{val1}
(a_{nr}\mid\chi)\in\mathcal{O}_{F^c}^\times\hspace{1cm}\mbox{for all }\chi\in\widehat{G}
\end{equation}
as a consequence of Proposition~\ref{NBG} (b). From Proposition~\ref{factor} (b), we then obtain an element 
$a'\in A_h$ such that $A_h=\mathcal{O}_FG\cdot a'$ and $\mathbf{r}_{G}(a')=\mathbf{r}_{G}(a_{nr})\mathbf{r}_{G}(a_{tot})$. Since $A_h=\mathcal{O}_FG\cdot a$ also, we have $a=\gamma\cdot a'$ for some $\gamma\in(\mathcal{O}_FG)^\times$, so
\[
(a\mid\chi)=\gamma(\chi)(a_{nr}\mid\chi)(a_{tot}\mid\chi)\hspace{1cm}
\mbox{for all }\chi\in\widehat{G}.
\]
Clearly $\gamma(\chi)\in\mathcal{O}_{F^c}^\times$ for all $\chi\in\widehat{G}$. The above, together with (\ref{val2}) and (\ref{val1}), then implies that $(a\mid\chi)\in\mathcal{O}_{F^c}^\times$ for all $\chi\in\widehat{G}$ as well.
\end{proof}

%---  Decomposition of Wild Local Resolvends -------------------------
\section{Decomposition of Local Wild Resolvends}\label{wildresol}

In this section, let $F$ be a finite extension of $\mathbb{Q}_p$ and assume further that $G$ has odd order. As in Section~\ref{computeval}, let $\zeta=\zeta_p$ be the chosen primitive $p$-th root of unity in $F^c$. We will also use the same notation and conventions set up in Definitions~\ref{defch5} and~\ref{Rn}. 

%--- Construction of Local Normal Basis Generators ------------
\subsection{Construction of Local Normal Basis Generators}

The major ingredient of the proof of Theorem~\ref{decomp} is the following proposition (recall the notation from Section~\ref{StickelS} and cf. Theorem~\ref{units}). We remark that its proof is very similar \mbox{to that of} \cite[Proposition 13.2]{T}.

\begin{prop}\label{rrwildmax}Let $h\in\mbox{Hom}(\Omega_F,G)$ be wildly and weakly \mbox{ramified, and be} such that $[F^h:F]=p$. Then, there exists $a\in F_h$ such that $F_h=FG\cdot a$ and
\begin{enumerate}[(1)]
\item $r_G(a)=\Theta^t_*(g)$ for some $g\in\Lambda(FG)^\times$;
\item $(a\mid\chi)\in\mathcal{O}_{F^c}^\times$ for all $\chi\in\widehat{G}$.
\end{enumerate}
\end{prop}

In what follows, let $h\in\mbox{Hom}(\Omega_F,G)$ be as in Proposition~\ref{rrwildmax}. Note that $p$ must be odd since $G$ has odd order and $F^h/F$ is wildly ramified of degree $p$.

First of all, we will introduce the basic set up and some essential notation. Set $L:=F^h$. Note that there is a canonical isomorphism
\[
\mbox{Gal}(L(\zeta)/F)\simeq\mbox{Gal}(F(\zeta)/F)\times\mbox{Gal}(L/F)
\]
because $[L:F]$ and $[F(\zeta):F]$ are coprime. Let $n\in\mathbb{N}$ be the unique integer dividing $p-1$ such that $\mbox{Gal}(F(\zeta)/F)\simeq R_n$ and let $d\in\mathbb{F}_p$ denote the class represented by $(p-1)/n$. We will also fix a generator $\tau$ of $\mbox{Gal}(L/F)$ and let $\widetilde{\tau}$ be the element in $\mbox{Gal}(L(\zeta)/F(\zeta))$ which is identified with $\tau$ via the above isomorphism.
 
We summarize the set-up in the diagram below, where the numbers indicate the degrees of the extensions.
\[
\begin{tikzpicture}[baseline=(current bounding box.center)]
\node at (0,0) [name=F] {$F$};
\node at (5,1) [name=F'] {$F(\zeta)$};
\node at (0,2.5) [name=L] {$L$};
\node at (5,3.5) [name=L'] {$L(\zeta)$};
\node at (0,-2.5) [name=Qp] {$\mathbb{Q}_p$};
\node at (1.5,-2.5) {($p$ is odd)};
\node at (10.25,2) {$d:=(p-1)/n\mbox{ (mod $p$)}$};
\path[font=\small]
(Qp) edge node[auto] {} (F)
(F') edge node[auto] {$\frac{p-1}{n}$} (F)
(F) edge node[auto] {$p$} (L)
(L') edge node[above] {$\frac{p-1}{n}$} (L)
(L') edge node[auto] {$p$} (F');
\draw[bend left=25] (F) edge node[pos=0.6,auto] {$R_n$} (F');
\draw[bend left=45] (F) edge node[auto] {$\langle\tau\rangle$} (L)
(L') edge node[auto] {$\langle\widetilde{\tau}\rangle$} (F');
\end{tikzpicture}
\]
Next, as in the proof of Proposition~\ref{valtotram}, we may deduce from Proposition~\ref{Aexists} and \cite[Theorem 1.1]{HJ} that $A^h=\mathcal{O}_F\mbox{Gal}(L/F)\cdot\alpha'$ for some $\alpha'\in A^h$. Moreover, the map
$a'\in\mbox{Map}(G,F^c)$ given by
\[
a'(s):=
\begin{cases}
\omega(\alpha') &\mbox{if $s=h(\omega)$ for $\omega\in\Omega_F$}\\
0 & \mbox{otherwise}
\end{cases}
\]
is well-defined and we have $A_h=\mathcal{O}_FG\cdot a'$. We will define the desired generator $a\in F_h$ using the resolvents (\ref{resolventdef}) of $a'\in A_h$ (cf. the proof of Lemma~\ref{yiunit}).

\begin{definition}\label{yidef}For each $i\in\mathbb{F}_p$, define
\[
y_i:=\sum_{k\in\mathbb{F}_p}\tau^k(\alpha')\zeta^{-ik}.
\]
\end{definition}

\begin{lem}\label{yiunit}For all $i\in\mathbb{F}_p$, we have $y_i\in\mathcal{O}_{L(\zeta)}^\times$.
\end{lem}
\begin{proof}Let $\omega_\tau\in\Omega_F$ be a lift of $\tau$. Then, the isomorphism $\mbox{Gal}(L/F)\simeq h(\Omega_F)$ induced by $h$ identifies $\tau$ with $h(\omega_\tau)$. Define $t:=h(\omega_\tau)$ and notice that $t$ has order $p$ because $[L:F]=p$ by hypothesis. Now, given $i\in\mathbb{F}_p$, let $\chi\in\widehat{G}$ be any character such that $\chi(t)=\zeta^i$. Then, we have
\[
(a'\mid\chi)=\sum_{s\in h(\Omega_F)}a'(s)\chi(s)^{-1}=\sum_{k\in\mathbb{F}_p}\tau^k(\alpha')\chi(t)^{-k},
\]
which is equal to $y_i$. It then follows from Theorem~\ref{units} that $y_i\in\mathcal{O}_{L(\zeta)}^\times$.
\end{proof}

\begin{lem}\label{yiprop1}For all $i\in\mathbb{F}_p$, we have $\widetilde{\tau}(y_i)=\zeta^iy_i$ and $y_i^p\in F(\zeta)^\times$.
\end{lem}
\begin{proof}Given $i\in\mathbb{F}_p$, the claim that $\widetilde{\tau}(y_i)=\zeta^iy_i$ follows from a simple calculation. Using this, we further deduce that
\[
N_{L(\zeta)/F(\zeta)}(y_i)=\prod_{k\in\mathbb{F}_p}\widetilde{\tau}^k(y_i)=\prod_{k\in\mathbb{F}_p}\zeta^{ik}y_i=y_i^p.
\]
Thus, indeed $y_i^p\in F(\zeta)^\times$, and this proves the lemma.
\end{proof}

Recall Definitions~\ref{defch5} and~\ref{Rn}. Consider the element
\[
\alpha:=\frac{1}{p}\Bigg(\sum_{k\in\mathbb{F}_p}\prod_{i\in R_n}y_i^{c(i^{-1}k)}\Bigg)\\
=\frac{1}{p}\Bigg(1+\prod_{i\in R_n}y_i^{c(i^{-1})}+\cdots+\prod_{i\in R_n}y_{i}^{c(i^{-1}(p-1))}\Bigg).
\]
We will show that the map $a\in\mbox{Map}(G,F^c)$ defined by
\begin{equation}\label{awildmax}
a(s):=\begin{cases}
\omega(\alpha) & \mbox{if }s=h(\omega)\mbox{ for }\omega\in\Omega_F\\
0 & \mbox{otherwise}
\end{cases}
\end{equation}
is well-defined and that it has the desired properties.

\begin{definition}\label{omegaimax}For each $i\in R_n$, define
\[
\omega_i\in\mbox{Gal}(L(\zeta)/L);\hspace{1em}\omega_i(\zeta):=\zeta^{i}
\]
(note that our notation here is slightly different from that used in \cite[Definition 13.4]{T}). Clearly, we have
\begin{equation}\label{omegaiyj}
\omega_i(y_j)=y_{ij}\hspace{1cm}\mbox{for all $j\in\mathbb{F}_p$}.
\end{equation}
\end{definition}

First, we show that $\alpha\in L$, which will then imply that $a$ is well-defined.

\begin{lem}\label{inLmax}We have $\alpha\in L$.
\end{lem}
\begin{proof}Clearly $\alpha\in L(\zeta)$ because $y_i\in L(\zeta)$ for all $i\in\mathbb{F}_p$ by definition. Hence, we have $\alpha\in L$ if and only if $\alpha$ is fixed by the action of $\mbox{Gal}(L(\zeta)/L)$. \mbox{Now, an} element of $\mbox{Gal}(L(\zeta)/L)$ is equal to $\omega_j$ for some $j\in R_n$. Then, \mbox{using (\ref{omegaiyj}), we} deduce that for each $k\in\mathbb{F}_p$, we have
\[
\omega_j\left(\prod_{i\in R_n}y_i^{c(i^{-1}k)}\right)
=\prod_{i\in R_n}y_{ij}^{c(i^{-1}k)}
=\prod_{i\in R_n}y_{i}^{c(i^{-1}jk)}.
\]
This implies that $\omega_j$ permutes the summands
\[
1,\prod_{i\in R_n}y_{i}^{c(i^{-1})},\dots,\prod_{i\in R_n}y_{i}^{c(i^{-1}(p-1))}
\]
in the definition of $\alpha$ and hence fixes $\alpha$. Thus, indeed $\alpha\in L$.
\end{proof}

Next, we compute the Galois conjugates of $\alpha$ in $L/F$.

\begin{prop}\label{conjugatesmax}For all $j,k\in\mathbb{F}_p$, we have
\[
\widetilde{\tau}^{j}\left(\prod_{i\in R_n}y_i^{c(i^{-1}k)}\right)=\zeta^{jkd}\cdot\prod_{i\in R_n}y_i^{c(i^{-1}k)}.
\]
In particular, this implies that for all $j\in\mathbb{F}_p$, we have
\[
\tau^j(\alpha)=\frac{1}{p}\sum_{k\in\mathbb{F}_p}\left(\zeta^{jkd}\prod_{i\in R_n}y_i^{c(i^{-1}k)}\right).
\]
\end{prop}
\begin{proof}Let $j,k\in\mathbb{F}_p$ be given. Notice that $\widetilde{\tau}^j(y_1)=\zeta^j y_1$ by Lemma~\ref{yiprop1}, and that $y_i=\omega_i(y_1)$ for all $i\in R_n$ by (\ref{omegaiyj}). Because $\mbox{Gal}(L(\zeta)/F)$ is abelian, for each $i\in R_n$, we have that
\begin{align*}
\widetilde{\tau}^{j}(y_i)
&=(\widetilde{\tau}^j\circ\omega_i)(y_1)\\
&=(\omega_i\circ\widetilde{\tau}^j)(y_1)\\
&=\omega_i(\zeta^{j}y_1)\\
&=\zeta^{ij}y_i.
\end{align*}
It then follows that
\begin{align*}
\widetilde{\tau}^j\left(\prod_{i\in R_n}y_i^{c(i^{-1}k)}\right)
&=\prod_{i\in R_n}\zeta^{ijc(i^{-1}k)}y_i^{c(i^{-1}k)}\\
&=\prod_{i\in R_n}\zeta^{jk}y_i^{c(i^{-1}k)}\\
&=\zeta^{jk(p-1)/n}\cdot\prod_{i\in R_n}y_i^{c(i^{-1}k)}.
\end{align*}
Since $(p-1)/n$ represents $d\in\mathbb{F}_p$ by definition, the claim now follows.
\end{proof}

Finally, we define the desired element $g\in\Lambda(FG)^\times$ (recall (\ref{lambda})).  As in the 

\noindent proof of Lemma~\ref{yiunit}, let $\omega_\tau\in\Omega_F$ be a lift of $\tau$ and set $t:=h(\omega_\tau)$, which has order $p$. It will also be helpful to recall Definition~\ref{cyclotomic}.

\begin{lem}\label{gmax}
The map $g\in\mbox{Map}(G(-1),(F^c)^\times)$ given by
\[
g(s):=
\begin{cases}
y_i^p&\mbox{if }s=t^{d^{-1}i^{-1}}\mbox{ for }i\in R_n\\
1&\mbox{otherwise}
\end{cases}
\]
is well-defined and preserves the $\Omega_F$-action. Thus, we have $g\in\Lambda(FG)^\times$.
\end{lem}
\begin{proof}
It is clear that $g$ is well-defined because $t$ has order $p$. To show that $g$ preserves the $\Omega_F$-action, let $\omega\in\Omega_F$ and $s\in G(-1)$ be given. 

If $s=t^{d^{-1}i^{-1}}$ for some $i\in R_n$, then $s$ has order $p$ and so $\omega\cdot s$ is determined by the action of $\omega$ on $\zeta$. Let $j\in R_n$ be such that $\omega|_{F(\zeta)}=\omega_j|_{F(\zeta)}$. Then, we have $\omega^{-1}(\zeta)=\zeta^{j^{-1}}$ and so $\omega\cdot s=s^{j^{-1}}=t^{d^{-1}i^{-1}j^{-1}}$. Recall that $y_i^p\in F(\zeta)$ by Lemma~\ref{yiprop1} and that $y_{ij}=\omega_j(y_i)$ by (\ref{omegaiyj}). We then deduce that
\[
g(\omega\cdot s)=y_{ij}^p=\omega_j(y_i^p)=\omega(g(s)).
\]
Now, if $\omega\cdot s= t^{d^{-1}i^{-1}}$ for some $i\in R_n$, then the same argument above shows that $s=\omega^{-1}\cdot(\omega\cdot s)=t^{d^{-1}i^{-1}j^{-1}}$ for some $j\in R_n$ as well. Hence, if $s\neq t^{d^{-1}i^{-1}}$ for all $i\in R_n$, then the same is true for $\omega\cdot s$. In this case, we have
\[
g(\omega\cdot s)=1=\omega(1)=\omega(g(s)).
\]
Hence, indeed $g$ preserves the $\Omega_F$-action, and $g\in\Lambda(FG)^\times$ by definition.
\end{proof}

\begin{proof}[Proof of Proposition~\ref{rrwildmax}]Let $a\in\mbox{Map}(G,F^c)$ and $g\in\Lambda(FG)^\times$ be as in (\ref{awildmax}) and Lemma~\ref{gmax}, respectively. Since $\alpha\in L$ by Lemma~\ref{inLmax} and $L=F^h$, it is clear that $a$ is well-defined and that $a\in F_h$. 

First, we will use the identification $\mathcal{H}(FG)=\mbox{Hom}_{\Omega_F}(S_{\widehat{G}},(F^c)^\times)$ in (\ref{idenH}) to show that $r_G(a)=\Theta^t_*(g)$. To that end, let $\chi\in\widehat{G}$ be given and let $k\in\mathbb{F}_p$ be such that $\chi(t)=\zeta^{k}$. Observe that $\langle\chi,t^{d^{-1}i^{-1}}\rangle_*=c(d^{-1}i^{-1}k)/p$ for all $i\in R_n$ by Definitions~\ref{Stickel} and~\ref{defch5}, and so
\[
\Theta_{*}^{t}(g)(\chi)
=\prod_{s\in G}g(s)^{\langle\chi,s\rangle_*}
=\prod_{i\in R_n}y_i^{c(d^{-1}i^{-1}k)}.
\]
On the other hand, because $\tau$ is identified with $t:=h(\omega_\tau)$ via the isomorphism $\mbox{Gal}(L/F)\simeq h(\Omega_F)$ induced by $h$, we have that
\[
(a\mid\chi)=\sum_{s\in h(\Omega_F)}a(s)\chi(s)^{-1}
=\sum_{j\in\mathbb{F}_p}\tau^{j}(\alpha)\zeta^{-jk}.
\]
Then, using Proposition~\ref{conjugatesmax}, we obtain
\begin{align*}
(a\mid\chi)
&=\frac{1}{p}\sum_{j\in\mathbb{F}_p}\sum_{l\in\mathbb{F}_p}\Bigg(\zeta^{jld}\prod_{i\in R_n} y_i^{c(i^{-1}l)}\Bigg)\zeta^{-jk}\\
&=\frac{1}{p}\sum_{l\in\mathbb{F}_p}\Bigg(\prod_{i\in R_n} y_i^{c(i^{-1}l)}\sum_{j\in\mathbb{F}_p}\zeta^{j(ld-k)}\Bigg)\\
&=\prod_{i\in R_n}y_i^{c(i^{-1}d^{-1}k)}.
\end{align*}
Hence, indeed $r_{G}(a)=\Theta_{*}^{t}(g)$. Since $y_i\in\mathcal{O}_{L(\zeta)}^\times$ for each $i\in R_n$ by Lemma~\ref{yiunit}, the above also shows that $(a\mid\chi)\in\mathcal{O}_{L(\zeta)}^\times$ for all $\chi\in\widehat{G}$. Via the identification in (\ref{iden1}), this implies that $\mathbf{r}_G(a)\in (F^cG)^\times$ as well, and thus $F_h=FG\cdot a$ by Proposition~\ref{NBG} (a). This proves that the element $a$ in (\ref{awildmax}) satisfies all of the properties claimed in Proposition~\ref{rrwildmax}.
\end{proof}

%--- Proof of Theorem ------------
\subsection{Proof of Theorem~\ref{decomp}}

\begin{proof}[Proof of Theorem~\ref{decomp}]
Let $h=h^{nr}h^{tot}$ be a factorization of $h$ as in Proposition~\ref{factor} (a). Since $F^{h^{tot}}/F$ is wildly, weakly, and totally ramified, the Galois group $\mbox{Gal}(F^{h^{tot}}/F)$ has exponent $p$ by Proposition~\ref{Aexists}.

Since $h^{tot}(\Omega_F)\simeq\mbox{Gal}(F^{h^{tot}}/F)$, we have
\begin{equation}\label{directprod}
h^{tot}(\Omega_F)=H_1\times \cdots\times H_r
\end{equation}
for subgroups $H_1,\dots,H_r$ each of order $p$. For each $i=1,\dots,r$, define
\[
h_i\in\mbox{Hom}(\Omega_F,G);\hspace{1em}h_i(\omega):=\uppi_i(h^{tot}(\omega)),
\]
where $\uppi_i:h^{tot}(\Omega_F)\longrightarrow H_i$ is the canonical projection map given by (\ref{directprod}). It is clear that $h^{tot}=h_1\cdots h_r$. For each $i=1,\dots,r$, observe that $F^{h_i}\subset F^{h^{tot}}$ and $[F^{h_i}:F]=p$. Since a Galois subextension of a weakly ramified extension is still weakly ramified (see \cite[Proposition 2.2]{V}, for example), each $h_i$ is wildly and weakly ramified. Hence, Proposition~\ref{rrwildmax} applies and there exists
 $a_i\in F_{h_i}$ with $F_{h_i}=FG\cdot a_i$ such that
\[
r_G(a_i)=\Theta^t_*(g_i)\hspace{1cm}\mbox{for some }g_i\in\Lambda(FG)^\times
\]
and that $(a_i\mid\chi)\in\mathcal{O}_{F^c}^\times$ for all $\chi\in\widehat{G}$. 

On the other hand, by a classical theorem of E. Noether (alternatively, by \cite[Proposition 5.5]{M}), there exists $a_{nr}\in\mathcal{O}_{h^{nr}}$ such that $\mathcal{O}_{h^{nr}}=\mathcal{O}_FG\cdot a_{nr}$. In addition, we know from (\ref{rrunram}) that
\[
r_G(a_{nr})=u\hspace{1cm}\mbox{for some }u\in\mathcal{H}(\mathcal{O}_FG)
\]
and that $(a_{nr}\mid\chi)\in\mathcal{O}_{F^c}^\times$ for all $\chi\in\widehat{G}$. Now, since $\mathbf{r}_G$ is bijective, there exists $a'\in\mbox{Map}(G,F^c)$ such that
\begin{equation}\label{a'}
\mathbf{r}_G(a')=\mathbf{r}_G(a_{nr})\mathbf{r}_G(a_1)\cdots\mathbf{r}_G(a_r).
\end{equation}
From (\ref{resol1}) and Proposition~\ref{NBG} (a), in fact we have $a'\in F_h$ and $F_h=FG\cdot a'$. But $F_h=FG\cdot a$ also, so we have $a=\gamma\cdot a'$ for some $\gamma\in(FG)^\times$, and
\[
r_G(a)=rag(\beta)r_G(a')=rag(\gamma)u\Theta^t_*(g),
\]
where $g:=g_1\cdots g_r\in\Lambda(FG)^\times$. It remains to show that $\gamma\in\mathcal{M}(FG)^\times$.

To that end, observe that 
\[
(FG)^\times=\mbox{Map}_{\Omega_F}(\widehat{G},(F^c)^\times)
\]
from the identification in (\ref{iden1}), and so
\[
\mathcal{M}(FG)^\times=\mbox{Map}_{\Omega_F}(\widehat{G},\mathcal{O}_{F^c}^\times).
\]
Now, for any $\chi\in\widehat{G}$, by definition we have
\[
(a\mid\chi)=\gamma(\chi)(a_{nr}\mid\chi)(a_1\mid\chi)\cdots (a_r\mid\chi).
\]
We know that $(a_{nr}\mid\chi),(a_1\mid\chi),\dots,(a_r\mid\chi)\in\mathcal{O}_{F^c}^\times$. Since $(a\mid\chi)\in\mathcal{O}_{F^c}^\times$ by Theorem~\ref{units}, it follows that $\gamma(\chi)\in\mathcal{O}_{F^c}^\times$ as well. Thus, indeed $\gamma\in\mathcal{M}(FG)^\times$ and this proves the theorem.
\end{proof}

\end{document}